\documentclass{article}[12pt]

\usepackage{fullpage, verbatim,graphicx,hyperref}
\usepackage[stable]{footmisc}
\usepackage{amsmath, amssymb, amsthm, amsfonts}

\newtheorem{theorem}{Theorem}[section]
\newtheorem{corollary}[theorem]{Corollary}
\newtheorem{lemma}[theorem]{Lemma}
\theoremstyle{definition}
\newtheorem{example}[theorem]{Example}

\newtheorem{conjecture}[theorem]{Conjecture}

\newcommand{\R}{\mathbb{R}}
\newcommand{\C}{\mathbb{C}}
\newcommand{\N}{\mathbb{N}}
\newcommand{\D}{\mathbb{D}}
\newcommand{\dbar}{\overline{\partial}}
\newcommand{\Bal}{\mathrm{Bal}}
\renewcommand{\Re}{\mathrm{Re}\,}

\author{Charles Z. Martin}
\date{\today}
\title{Variational Methods and Planar Elliptic Growth}

\begin{document}

\maketitle
\begin{abstract}
\noindent
  A nested family of growing or shrinking planar domains is called a Laplacian growth process if the normal velocity
  of each domain's boundary is proportional to the gradient of the domain's Green function with a fixed singularity on the interior.
  In this paper we review the Laplacian growth model and its key underlying assumptions,
  so that we may consider a generalization to so--called elliptic growth, wherein the Green function is replaced with that of a
  more general elliptic operator---this models, for example, inhomogeneities in the underlying plane.
  In this paper we continue the development of the underlying mathematics for elliptic growth, considering
  perturbations of the Green function due to those of the driving operator,
  deriving characterizations and examples of growth,
  developing a weak formulation of growth via balayage, and discussing of a couple of inverse
  problems in the spirit of Calder\'on. We conclude with a derivation of a more delicate, reregularized model for Hele--Shaw flow.
\end{abstract}

\section{Introduction}
  \subsection{Hele--Shaw Flow and Laplacian Growth}
  A slow, viscous, incompressible fluid is trapped in a narrow region between two parallel plates.
  With one dimension significantly smaller
  in scale than the others, one might imagine neglecting fluid depth and treating the flow as two--dimensional.
  When the gap between plates is sufficiently small, we can indeed make this approximation;
  the flow behaves like that of a two--dimensional
  fluid in a porous medium, obeying Darcy's law. The resulting dynamical law is
  \begin{equation*}
    \mathbf{v} = \frac{-h^2}{12\mu}\nabla p,
  \end{equation*}
  where $p$ is the fluid pressure, $h$ is the gap width between the plates, and $\mu$ is the viscosity of the fluid.
  This model of fluid flow is known as Hele--Shaw flow, after English engieneer Henry Selby Hele--Shaw who first observed the
  phenomenon.
  
  Let's now consider the fluid domain holistically; let $D\subseteq \C$ be the fluid domain, which we assume to be bounded.
  Since only the normal boundary velocity $v_n$ is observable, the dynamics reduce to
  \begin{equation}\label{Hele-shaw-darcy}
    v_n(\zeta) = \frac{-h^2}{12\mu}\partial_n p(\zeta) \qquad (\zeta\in\partial D).
  \end{equation}
  Next we consider a boundary condition for the fluid pressure. Assume that the ambient fluid, say air or water,
  is comparatively nonviscous and at a constant, atmospheric pressure $p_0$ near the fluid domain $D$.
  At the fluid boundary the force balance is
  \begin{equation*}
    p = p_0 + \kappa\sigma,
  \end{equation*}
  where $\kappa$ denotes the mean curvature of $\partial D$ and $\sigma$ is a surface tension coefficient.
  In the derivation of Hele--Shaw flow one finds that the vertical (that is, orthogonal to the plates) dependence of velocity is
  quadratic and the same at all points on the boundary;
  since the plate separation is so small, the curvature of the boundary in the vertical direction is very large compared
  to the curvature in the plane. As such, the mean curvature $\kappa$ is essentially constant at all points of the fluid boundary.
  Finally, adding an appropriate constant allows us to take
  \begin{equation} \label{pressure_vanish}
    p\Big|_{\partial D} = 0.
  \end{equation}
  That said, the neglect of surface tension is source of much discussion on the mathematical properties of Hele--Shaw flow;
  we will return to this point in the section on reverse--time behavior below.
  
  An interesting consequence of a two--dimensional flow existing in three--dimensional space is that we can consider
  extracting or injecting fluid into the system from above. Supposing that fluid is pumped at a point source $w\in D$, we can
  adjust the model accordingly by altering the incompressibility condition to
  \begin{equation} \label{Hele-shaw_divergence}
    \nabla\cdot \mathbf{v} = Q\delta_w,
  \end{equation}
  where $\delta_w$ is a Dirac delta supported at $w$ and $Q$ is a real parameter controlling the rate of suction or
  injection---negative and positive $Q$ yielding suction and injection, respectively. Combining equations
  \eqref{Hele-shaw-darcy}, \eqref{pressure_vanish} and \eqref{Hele-shaw_divergence}, we have the following model:
  \begin{equation}
    \begin{cases}
      \nabla\cdot \mathbf{v} = Q\delta_w & \textrm{ in } D\\
      p=0 & \textrm{ on } \partial D \\
      v_n = -(h^2/12\mu)\partial_n p & \textrm{ on } \partial D.
    \end{cases}
  \end{equation}
  We can rewrite this a bit if we recall that $\mathbf{v} = -(h^2/12\mu)\nabla p$ throughout $D$;
  using this and---for simplicity---taking $h^2/12\mu=1$, we obtain
  \begin{equation}\label{LG_definition}
    \begin{cases}
      \Delta p = -Q\delta_w & \textrm{ in } D\\
      p=0 & \textrm{ on } \partial D \\
      v_n = -\partial_n p & \textrm{ on } \partial D.
    \end{cases}
  \end{equation}
  A growing or shrinking family of domains obeying these equations is known as a Laplacian growth process.
  Though it might seem strange to give the flow a second name, Laplacian growth can in fact model
  a number of other physical processes, such as electrodeposition and crystal formation, and stochastic processes, such as
  diffusion--limited aggregation. In the next two sections
  we will discuss the mathematics of such Laplacian growth, emphasizing both its elegant and pathological properties.
  There is much literature on Laplacian growth and Hele--Shaw flows;
  surveys of both the history and physics can be found in \cite{Gustafsson},
  \cite{Mineev}, and \cite{Hele-Shaw_History}.
  One of the major breakthroughs in the subject was the introduction of conformal mapping techniques to model boundary motion, which
  first appeared in articles by Polubarinova \cite{Polubarinova, Polubarinova2} and Galin \cite{Galin}. Studies of uniqueness of forward--time
  solutions first appeared in work by Kufarev and Vinogradov in \cite{Kufarev}, and have since been refined and simplified
  \cite{Gustafsson_uniqueness, Reissig}. Reverse--time behavior is more delicate and is discussed below.
  There have been many numerical treatments \cite{Hou, Ceniceros, Ceniceros2},
  but reverse--time modelling is more difficult due to the aforementioned delicacy
  of the problem \cite{Meiburg}. Other descriptions of the dynamics use the so--called Schwarz function $S$ of $\partial D$, which is
  an analytic function in a neighborhood of the boundary so that $S(z) = \overline{z}$ on $\partial D$. The Schwarz function---which
  does not appear in the works of Hermann Schwarz---was named in his honor by Paul Davis, who developed
  the idea in his 1974 monograph \cite{Davis}. Laplacian growth can then be
  reduced to the equation
  \begin{equation*}
    \partial_t S + 2\partial_z W = 0,
  \end{equation*}
  where $W$ is an analytic function whose real part is negative pressure;
  a derivation and discussion can be found in, e.g. \cite{Mineev, Gustafsson}
  Finally, notice that $-p/Q$ is the Green function
  of the domain $D(t)$. Indeed, Laplacian growth has deep connections with
  problems and results in potential theory.
  
  Now we turn to some of the elegant structure one can find within the Laplacian growth model.
  The most well--known and fundamental result is Richardson's theorem \cite{Richardson}:
  \begin{theorem}
    Suppose $\{D(t):0\leq t< t_0\}$ is a family of domains with $C^2$ boundaries satisfying \eqref{LG_definition}. If $f:\C\to\C$ is
    harmonic, then
    \begin{equation*}
      \frac{d}{dt}\int_{D(t)} f\, dA = Qf(w).
    \end{equation*}
    In particular, if $w=0$ we have
    \begin{align} \label{harmonic_moments}
      \frac{d}{dt}\int_{D(t)} z^n\, dA &= 0 \qquad (n\in \N), \\
      \frac{d}{dt}\mathrm{area}(D(t)) &= Q.
    \end{align}
  \end{theorem}
  Notice that we have now made explicit our previously tacit assumption on the regularity of $\partial D(t)$; so long as
  $\partial D(t)$ is smooth enough to have an outward normal velocity, application of Green's identity is valid as well.
  The integrals in \eqref{harmonic_moments} are called the harmonic moments of the domain $D(t)$.
  Richardson's theorem shows that these moments behave in a simple way during Laplacian growth; the domain evolves in
  such a way that the area increases
  linearly in time while the harmonic moments remain constant. Perhaps surprisingly, the harmonic moments locally characterize the process
  in the sense of an integrable system.
  
  There are a few different meanings of the term `integrable system,' some of which do not have widely accepted, precise definitions.
  Informally, a dynamical system is integrable if there exists an underlying structure which describes the system's evolution
  in a simple way, akin to the simplicity of linear ordinary differential equations. To avoid a lengthy and tangential discussion
  of integrability in general, let's instead describe what it means for Laplacian growth. For simplicity we will
  only consider simply--connected domains, as in \cite{Integrable1}; the interested reader can find a discussion of multiply--connected
  domains in \cite{Integrable2}.
  
  Consider the harmonic moments as generalized coordinates for the domain evolution and define
  \begin{equation}\label{moments}
    t_n = \int_{D(t)} z^n\,dA \qquad (n\in\N\cup\{0\}).
  \end{equation}
  Having assumed that $D(t)$ is simply connected, we have the result that a domain evolution keeping each $t_n$ constant must be trivial.
  In this sense the moments $\{t_n\}$ locally parametrize the phase space of domains. Furthermore,
  there is a commuting of flows in the following sense: if a domain $D$ undergoes evolution due to an infinitesimal change in $t_m$
  followed by an infinitesimal change in $t_n$, the resulting domain is the same had we instead
  evolved $D$ by increasing $t_n$ before $t_m$.
  In this way the moments $\{t_n\}$ are independent.
  Finally, the Laplacian growth of a domain is the simply the `flow line' in the $t_0$--direction which contains the domain.
  There can be more rigor to the argument than we are providing here; in \cite{Integrable1} the authors construct a Poisson bracket
  and Hamiltonian before showing that the generalized coordinates $\{t_n\}$ are a maximal set and are in involution which each other.
  That is, Laplacian growth is integrable in the classical Hamiltonian sense.
  
  A disk undergoing Laplacian growth with a source at its center evolves into a larger disk---this is clear from the symmetry of the
  situation. Now consider a domain whose boundary is a hypocycloid, as seen in figure \ref{fig:hypocycloid}.
  \begin{figure}[b]
  \centering
  \includegraphics[width=0.5\linewidth]{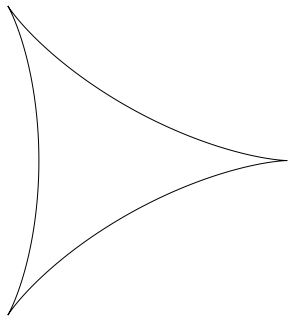}
  \caption{Hypocycloid}
  \label{fig:hypocycloid}
  \end{figure}
  If this domain undergoes Laplacian growth
  with the source at the center, eventually the domain evolves into a disk as well (see \cite{Mineev}). In reverse time, this
  demonstrates ill--posedness; a disk undergoing Laplacian growth with a sink at its center can evolve into either a smaller
  disk or a hypocycloid. That is, the dynamics are not unique in the case of shrinking domains. Furthermore, the formation of cusps
  in finite time can mean a deviation from the physical system being modeled. The situation is unfortunately common in this sense;
  Laplacian growth with suction will cause most initial domains to similarly form cusps.
  
  We should remark that cusps are not always bad. Some processes modeled by Laplacian growth do exhibit
  cusp formation in an experimental setting, such as the dynamics of microstructure of materials.
  Furthermore, for some types of cusps the solution can continue to exist \cite{Howison}.
  That said, there are other processes---such as Hele--Shaw flow---for which
  cusps are non--physical. Thus we are led to seek generalizations or corrections to the model.
  
  A common correction to the model is the reinsertion of surface tension effects
  \cite{McLean}. Recall that we neglected surface tension
  using the premise that the mean curvature at each point of the boundary is dominated by the vertical contribution, but
  in a neighborhood of a cusp this approximation is no longer valid. With surface tension Laplacian growth no longer forms cusps,
  but it also loses integrability (see \cite{Mineev}). Furthermore, while surface tension helps correct the model for Hele--Shaw flow,
  it is not necessarily appropriate for other situations.
  
  Let's return to the derivation of Laplacian growth dynamics again, but this time we will reexamine our assumption on $\lambda$,
  the permeability of the underlying plane. We assumed $\lambda$ to be constant everywhere, but again this approximation
  will fail in the vicinity of a cusp---unless the permeability is truly constant, but this possibility is non--physical.
  Taking $\lambda$ to be a nonconstant scalar function changes the derivation in two ways: the incompressibility condition becomes
  $\nabla\cdot (\lambda\mathbf{v}) = \delta_w$, and Darcy's law becomes $\mathbf{v} = -\lambda \nabla p$.
  Thus we are led to the following generalization:
  \begin{equation*}
    \begin{cases}
      \nabla\lambda\nabla p = -Q\delta_w & \textrm{ in } D\\
      p=0 & \textrm{ on } \partial D \\
      v_n = -\lambda\partial_n p & \textrm{ on } \partial D,
    \end{cases}
  \end{equation*}
  where $\nabla\lambda\nabla$ is an abbreviation for $\nabla\cdot (\lambda\nabla)$. Since the dynamics now rely upon the
  elliptic operator $\nabla\lambda\nabla$---commonly known as a Laplace--Beltrami operator---we call these dynamics elliptic growth.
  The majority of this paper is dedicated to characterizing elliptic growth and studying the way in which it generalizes the better--known
  Laplacian growth model.
  
  As a final aside, we remark that other renormalizations and generalizations of Laplacian growth exist.
  One such example is kinetic undercooling regularization \cite{Hohlov}, which replaces the condition boundary condition $p=0$ with
  the more general
  \begin{equation*}
    \beta \partial_n p + p = 0
  \end{equation*}
  on $\partial D$, where $\beta \geq 0$ is constant.
  For another example,
  quasi--2D Stokes flow
  is an attempt to retain the three--dimensional nature of a two--dimensional Hele--Shaw flow; the resulting
  renormalization behaves as an intermediate type of flow, between those of Stokes and Hele--Shaw. A brief discussion and derivation
  are given in the final section of this paper.

\subsection{Balayage and Weak Laplacian Growth}

Newton's theory of gravitation was a landmark achievement of science, but for every physical question it answered a mathematical question
took its place. For this reason, Newton spends time in his seminal work demonstrating a few ways one can sidestep the resulting mathematical
problems. For example, a uniformly dense spherical shell exerts no gravitational force within its interior cavity
and produces the same field as a point mass on its exterior; this fact allows one to discuss the orbits of planets without having to consider
their spatial extent. This result in particular is the simplest example of a general concept in potential theory known as
balayage, the notion of sweeping a measure outward while not changing its far field potential.

Our main goal in this section is to define balayage of a measure and describe a few of its properties, but first we should
review a few of the fundamental objects in measure theoretic potential theory. We consider only finite and positive Borel measures
that are compactly supported on $\C$---such a measure is thought of as a mass or charge distribution. Associated to
a measure $\mu$ is its (Newtonian) potential
\begin{equation*}
  U^\mu(z) = \int (2\pi)^{-1}\ln|z-w|\, d\mu(w).
\end{equation*}
Note that we can reobtain the measure via $\mu = \Delta U^\mu$, and as such the correspondence $\mu\leftrightarrow U^\mu$
can be viewed as a sort of duality between measures and functions. Any change in $\mu$ is reflected as a change in $U^\mu$,
but we might ask if we can alter $\mu$ in a way that leaves $U^\mu$ unchanged outside of a bounded domain---that is, leaving the
potential unchanged `far away'. This is perhaps most familiar in the setting of electrostatics; a charge distribution placed
on a perfect conductor will rearrange itself to be supported on the boundary while not changing the potential outside of the conductor.
An alteration of a measure in this way is the classical notion of balayage, but we will consider a somewhat more general version.

Suppose we have a measure $\mu$ which is rather `dense,' such as a Dirac delta. Rather than sweeping $\mu$ completely out of a given domain,
perhaps we only wish to lessen its density, in the following sense. Given another measure $\tau$, we seek a measure $\nu$ so that
$\nu\leq \tau$, while $\nu$ is somehow as close to $\mu$ as possible. Ideally closeness would be in terms of an energy norm, but
$\mu$ might have infinite energy and in two dimensions the energy `norm' need not even be positive. To get around this problem we will frame
balayage as an obstacle problem for the corresponding potential functions.

We now give a precise definition of (partial) balayage, as it appears in \cite{Balayage}. Let
$\tau$ be an arbitrary measure on $\C$, without the aforementioned restrictions we gave for mass distributions. If $\mu$
is a mass distribution then we consider the following obstacle problem: find the smallest function $V$ so that both
\begin{equation}\label{balayage_definition}
  \begin{cases}
    V \geq U^\mu \\
    \Delta V \leq \tau
  \end{cases}
\end{equation}
throughout $\C$. The balayage of $\mu$ with repsect to $\tau$ is then $\Bal(\mu,\tau) = \Delta V$.
The existence of a solution is well known (see \cite{Balayage2}).
Note that this problem reduces to that of classical balayage when $\tau = 0$ in a domain $D$ and $\tau = \infty$ on $\C\setminus D$.

Balayage is intimately connected with Laplacian growth, as we discuss below. To this end, we will need the following two
theorems; proofs can be found in \cite{Balayage}.
\begin{theorem}\label{thm:balayage1}
  Suppose that $D(0)\subset\C$ is a bounded domain and $w\in D(0)$.
  Then for $t\geq 0$ there exists a domain $D(t)\supseteq D(0)$ for which
  \begin{equation*}
    \Bal(\chi_{D(0)} + t\delta_w,1) = \chi_{D(t)},
  \end{equation*}
  where 1 is shorthand for Lebesgue measure and $\chi_D$ denotes Lebesgue measure restricted to $D$.
\end{theorem}
\noindent
From our previous discussions, one should sense that Laplacian growth is closely tied to potential theory. Since balayage
is the `bleeding' of a measure like a porous fluid flow, the following theorem is perhaps unsurprising.
\begin{theorem}
  Let $\{D(t):0\leq t<t_0\}$ be a Laplacian growth process with $Q>0$ as defined in \eqref{LG_definition}. Then
  \begin{equation}\label{LG_weak}
    \Bal(\chi_{D(0)} + Qt\delta_w, 1) = \chi_{D(t)}.
  \end{equation}
\end{theorem}
\noindent
This theorem leads us to extending the definition of Laplacian growth; Balayage always exists for any domain $D(0)$, and by theorem
\ref{thm:balayage1} the expression $\Bal(\chi_{D(0)} + Qt\delta_w,1)$ is always the characteristic function of a domain. Hence
we call a family of domains satisfying \eqref{LG_weak} a weak solution of Laplacian growth. In this way
we can study growth of a domain that doesn't have a smooth enough boundary to support a normal boundary velocity function.

The weak formulation of Laplacian growth allows us to draw conclusions about the strong formulation. For example, writing
$\Bal(\mu,1)=\chi_D$ identifies $D$ uniquely up to null sets, so forward--time Laplacian growth in the weak sense has a unique solution
(once we choose a normalization for $D$). Therefore forward--time Laplacian growth in the strong sense has a unique solution as well.
We will give a more thorough treatment of this weak formulation in a later section, when we generalize
it to elliptic growth processes.

\subsection{Variations of the Green Function}

The study of boundary value problems presents a significant computational challenge. In theory there are numerous formulas, transforms and methods for solving, say, the
Dirichlet problem on an ellipse for the Laplace--Beltrami operator. For example, knowledge of the Green function would reduce the
problem to that of calculating an integral. However, the computation of a Green function is prohibitively difficult; considering its central role
we might turn to methods of approximating it instead. One approach aims to replace the Green function of the problem
with that of an easier problem which is, in some sense, nearby. There are two ways this heuristic can proceed: variation of the
underlying domain and variation of the relevant operator. We first mention the more classical approach of domain variation due to
Hadamard. A rigorous treatment of this theorem can be found in \cite{SS} or in chapter 15 of \cite{Garabedian}.

\begin{theorem}[Hadamard's formula] \label{Hadamard}
  Let $p\in C(\partial D)$ be a positive function and suppose that for $\epsilon > 0$ each point $\zeta\in\partial D$ is moved along the outward normal direction
  a distance $\epsilon p(\zeta)$. The Green function $g^*$ of the new domain $D^*$ satisfies
  \begin{equation*}
    g_w^*(z) = g_w(z)  -\epsilon\int_{\partial D} p \,\partial_n g_z\cdot \partial_n g_w \,ds
      + o(\epsilon)
  \end{equation*}
  as $\epsilon \to 0$ for each fixed $z,w\in D$.
\end{theorem}

In traditional notation of the calculus of variations, we can write the aforementioned formula as
\begin{equation*}
	\delta g_w(z) :=\lim_{\epsilon\to 0}\frac{g_w^*(z)-g_w(z)}{\epsilon}
	= -\int_{\partial D} p \,\partial_n g_z\cdot \partial_n g_w \,ds.
\end{equation*}
We will at times use this notation for sake of clarity.
Many ideas and corollaries emerge from Hadamard's formula, such as:
\begin{enumerate}
  \item The outward normal derivative of the Green function is positive since, for instance, it is the density of the domain's harmonic measure.
  Therefore the variation is always negative, so we conclude that enlarging a domain decreases the Green function at every point.
  
  \item Let $\lambda :\C\to\R$ be a positive smooth function and define the operator $L = \nabla\lambda\nabla$. If we alter the definition of the Green function so that
  $Lg = \delta$, then we can derive another variational formula. From Green's identity
  \begin{equation*}
    \int_{\partial D} \lambda\left(u\frac{\partial v}{\partial n} - v\frac{\partial u}{\partial n}\right)\, ds = \int_{D} \left(uL v - vL u\right)\, dA,
  \end{equation*}
  we can derive the first variation of $g$:
  \begin{equation*}
    \delta g_w(z) = -\int_{\partial D} p\lambda\, \partial_n g_z \cdot \partial_n g_w\,ds.
  \end{equation*}
  
  \item One can also obtain variations for related functions, such as $\partial_n g_w(z)$, under domain variation.
  Hadamard's student Paul L\'evy spent some of his early career furthering his advisor's work in this direction,
  producing formulas such as
  \begin{equation*}
    \delta\partial_n g_w(z) = \mathrm{p.v.}\int_{\partial D} \left(\partial_n g_w(z)\delta n(z) - \partial_n g_w(\zeta) \delta n(\zeta) \right)
    \frac{\partial^2 g(z,\zeta)}{\partial n_z\partial n_\zeta}\, ds(\zeta).
  \end{equation*}
  A thorough discussion appears in the second chapter of part II in \cite{Levy}.
  
  \item
  Let $T_c$ denote a time variable whose increments are associated to pumping fluid at the point $c\in D$. Hadamard's formula yields
  \begin{equation*}
    \frac{\partial g(a,b)}{\partial T_c} = -\int_{\partial D(t)} \partial_n g_a \cdot \partial_n g_b \cdot\partial_n g_c\, ds,
  \end{equation*}
  which is clearly symmetric in the variables $a,b,c$. From here we have a so--called zero--curvature condition
  \begin{equation*}
    \frac{\partial g(a,b)}{\partial T_c} = \frac{\partial g(a,c)}{\partial T_b} = \frac{\partial g(b,c)}{\partial T_a},
  \end{equation*}
  which indicates a commutation of flows across various points of pumping.
  In \cite{Integrable2} the authors use these observations in their discussion of the integrable structure of Laplacian growth;
  from the zero--curvature relation they embed Laplacian growth of a multiply--connected domain
  into a hierarchy known as the Whitman equations.
  
\end{enumerate}

Given a conductor made from a standard material, we presumably know the conductivity function---and hence the
underlying operator---governing the electrostatic potential. If instead we had a material with impurities or an altogether new substance,
the conductivity function would be unknown. Thus for some applications, we might want to approximate the Green function not under
perturbations of the domain, but rather of the operator.
For example, if we turn our attention to inverse problems we might pose a question wherein an unknown operator has a Green function with
a prescribed property. The existence of the underlying operator can be studied by understanding
what sorts of variations in the Green function can result from perturbing a well--understood operator, such as the Laplacian.

Hadamard's theorem on the variation of a domain's Green function is based upon perturbations of the boundary.
Taking a different approach, we examine situations wherein the domain is fixed but rather the underlying operator is close to the Laplacian in various ways. The following variational formulas were derived, discussed and proved in \cite{Martin}.
\begin{theorem}\label{Schrodinger}
  Fix $w\in D$ and define the integral operator
  \begin{equation*}
    T\phi(z) = \int_D \phi g_z\, dA.
  \end{equation*}
  Suppose that $p$ is a smooth scalar function defined in a neighborhood of $D$
  with corresponding multiplication operator $P$.
  The Green function $g_w^*$ of the Schr\"odinger operator $\Delta-\epsilon p$ satisfies
  \begin{equation*}
     g^*_w = g_w + \epsilon TPg_w + o(\epsilon)
  \end{equation*}
  as $\epsilon \to 0$, where the convergence of $o(\epsilon)$ is uniform.
  Furthermore, a full series expansion is given by
  \begin{equation} \label{Schrodinger_series}
    g^*_w = \sum_{n=0}^\infty \epsilon^n (TP)^n g_w.
  \end{equation}
\end{theorem}
\begin{theorem}\label{Beltrami}
  Fix $w\in D$ and suppose that $p$ is a smooth scalar function in a neighborhood of $D$.
  Given $\epsilon > 0$ we can define $\lambda(z) = 1+\epsilon p(z)$. Then as $\epsilon \to 0$
  the Green function $g^*$ for $L=\nabla \lambda \nabla$ satisfies
  \begin{equation}
    g^*(z,w) = g(z,w) + \epsilon \int_D p(\xi) \nabla g(z,\xi) \cdot \nabla g(\xi,w)\, dA(\xi) + o(\epsilon), \label{Beltrami1}
  \end{equation}
  where all derivatives are with respect to $\xi$. Furthermore, the error term converges uniformly.
  An alternate formula is also true:
  \begin{equation}
    g^*(z,w) = g(z,w) -\epsilon g(z,w)\left(\frac{p(z) + p(w)}{2}\right) + \frac{\epsilon}{2}\int_D g_z g_w \Delta p \, dA + o(\epsilon). \label{Beltrami2}
  \end{equation}
\end{theorem}
The second theorem was derived from the first one via a lemma which will be helpful to us in our discussions below.
\begin{lemma}[Converting Laplace--Beltrami to Schr\"odinger]  \label{conversion_lemma}
  Let $\lambda > 0$ be a smooth scalar function in a neighborhood of $D$ and fix $w\in D$.
  Define the function $u=\lambda^{-1/2}\Delta\lambda^{1/2}$.
  If $g_w$ denotes the Green function $\nabla\lambda\nabla$, then for $z\in \overline{D}$ the function
  \begin{equation*}
    G_w(z) = g_w(z) \sqrt{\lambda(w)\cdot\lambda(z)}.
  \end{equation*}
  is a Green function for $\Delta - u$.
\end{lemma}

Given a Green function on a domain $D$, the outward normal derivative on $\partial D$ is of great importance,
both for boundary--value problems on $D$ as well as Laplacian growth. The normal derivative of $g_z$ is not much easier
to compute than $g_z$ itself, so we consider variational formulas for the derivative as we did for the Green function itself.
The following lemma was used in \cite{Martin2} for this purpose:

\begin{lemma}\label{NormalLemma}
  Let $D$ be a bounded domain in $\C$ with $C^1$ boundary and $f\in C^1(\overline{D})\cap C^2(D)$.
  and suppose further that $f=0$ on $\partial D$. Then for $\zeta\in \partial D$ we have
  \begin{equation*}
    \frac{\partial f}{\partial n}(\zeta) = \int_D \Delta f(z)P(z,\zeta)\, dA(z),
  \end{equation*}
  where $P(z,\zeta) = \partial_n g_z(\zeta)$ is the Poisson kernel of $D$.
\end{lemma}

Application of this lemma to various Green functions required that they extend across the boundary continuously. Therefore
the variational formulas below all require $\partial D$ to be smooth analytic, though perhaps this restriction could be relaxed
via other methods. These theorems use the (unusual) notation
\begin{equation*}
  P_\zeta(z) = P(\zeta,z) = \partial_n g_z(\zeta)
\end{equation*}
for the Poisson kernel of the domain, emphasizing the single variable dependence $z\mapsto P(\zeta,z)$.
\begin{theorem}\label{NormalTheorem}
  Fix $w\in D$, a bounded domain in $\C$ with smooth analytic boundary.
  Suppose that $u\in C^\infty(\overline{D})$ is a positive function.
  The outward normal derivative of the Green function $g^*$ of the Schr\"odinger operator $\Delta-\epsilon u$ satisfies
  \begin{equation*}
     \partial_n g^*_w(\zeta) = \partial_n g_w(\zeta) + \epsilon \int_D u g_wP_\zeta\, dA + o(\epsilon)
  \end{equation*}
  as $\epsilon \to 0$, where the convergence of $o(\epsilon)$ is uniform in $\zeta$.
\end{theorem}
\begin{theorem}
  Fix $w\in D$, a bounded domain in $\C$ with smooth analytic boundary.
  Suppose that $u\in C^\infty(\overline{D})$ is a positive function and define $\lambda(z) = 1+\epsilon u(z)$ for $\epsilon > 0$.
  The outward normal derivative of the Green function $g^*$ of the Laplace--Beltrami operator $L=\nabla\lambda\nabla$ satisfies
  \begin{equation*}
     \partial_n g^*_w(\zeta) =
     \partial_n g_w(\zeta) + \frac{\epsilon}{2}\left[\int_D \Delta u g_w P_\zeta\, dA - \partial_n g_w(\zeta)\left[u(\zeta)+u(w)\right]
      \right] + o(\epsilon)
  \end{equation*}
  as $\epsilon \to 0$, where the convergence of $o(\epsilon)$ is uniform in $\zeta$.
\end{theorem}

\section{Application to Boundary--Value Problems}

\subsection{The Dirichlet Problem for Schr\"odinger}
Before we turn to elliptic growth, we briefly consider an elementary application of the Green variation which has not yet
appeared in the literature.
Consider solving the Dirichlet problem on a domain $D$:
\begin{equation} \label{Dirichlet}
  \begin{cases}
    (\Delta-\epsilon u) \phi_\epsilon = 0 & \textrm{ in } D \\
    \phi_\epsilon = f &\textrm{ on } \partial D
  \end{cases}
\end{equation}
where $f$ is a bounded Borel function, $u\geq 0$ is smooth in a neighborhood of $\overline{D}$
and $\epsilon \geq 0$ is small. If we ignored the $\epsilon u$ term
in the problem---that is, assume $\epsilon$ is zero---we can estimate the error with the following theorem and corollary.
\begin{theorem}
  Let $D\subset \C$ be a bounded domain with smooth, analytic boundary and suppose $u\in C^\infty(\overline{D})$ is a positive function.
  For $\epsilon \geq 0$ and a bounded Borel function $f$ the solution $\phi_\epsilon$ to the Dirichlet problem \eqref{Dirichlet} satisfies
  \begin{equation*}
    \phi_\epsilon(z) = \phi_0(z) + \epsilon \int_D u \phi_0 g_z\, dA + o(\epsilon)
  \end{equation*}
  as $\epsilon \to 0$, where the error term converges uniformly in $z$.
\end{theorem}
\begin{proof}
  Let $g^*_w$ denote the Green function of the operator $\Delta-\epsilon u$ in the region $D$.
  Note that
  \begin{equation*}
    \begin{cases}
      (\Delta-\epsilon u)(\phi_\epsilon - \phi_0) = \epsilon u \phi_0 & \textrm{ in } D \\
      \phi_\epsilon - \phi_0 = 0 & \textrm{ on } \partial D,
    \end{cases}
  \end{equation*}
  whence an expression for $\phi_\epsilon$
  is given by
  \begin{equation*}
    \phi_\epsilon(z) = \phi_0(z) + \epsilon\int_D u\phi_0 g^*_z\, dA.
  \end{equation*}
  From the perturbation formula \eqref{Schrodinger_series} we have
  \begin{equation*}
    \phi_\epsilon(z) = \phi_0(z) + \epsilon\int_D u\phi_0 g_z\, dA + o(\epsilon),
  \end{equation*}
  as desired.
\end{proof}
\begin{corollary}
  With the same assumptions as the previous theorem, the linearization of $\phi_\epsilon$ has the pointwise bound
  \begin{equation*}
    |\delta\phi_\epsilon(z)| \leq \|u\|_2\|g_z\|_2 \|f\|_\infty.
  \end{equation*}
\end{corollary}
\begin{proof}
  From the maximum modulus principle for harmonic functions, $\sup_D |\phi_0| \leq \sup_{\partial D} |f|$. The result follows from this and the
  Cauchy-Schwarz inequality.
\end{proof}
The following lemma will ease the computation in the examples below.
\begin{lemma}\label{Green_area}
  Given a point $z\in \D$ and an integer $n\geq 0$ we have
  \begin{equation*}
    \int_\D |\xi|^{2n} g_z(\xi)\, dA(\xi) = -\frac{1-|z|^{2n+2}}{4(n+1)^2},
  \end{equation*}
  where $g_z$ denotes the Green function of the Laplacian on $\D$ with singularity at $z$.
\end{lemma}
\begin{proof}
  First we note that
  \begin{equation*}
    \Delta\left(\frac{|\xi|^{2n+2}}{4(n+1)^2}\right) = 4\partial\dbar\left(\frac{(\xi\overline{\xi})^{n+1}}{4(n+1)^2}\right)
      = |\xi|^{2n}.
  \end{equation*}
  We use the Poisson--Jensen formula
  \begin{equation*}
    v(z) = \int_{\partial \D} v\partial_n g_z\, ds + \int_\D g_z\Delta v\, dA
  \end{equation*}
  to write
  \begin{align*}
    \frac{|z|^{2n+2}}{4(n+1)^2}
    &= \int_{\partial \D} \frac{|\zeta|^{2n+2}}{4(n+1)^2}\partial_n g_z(\zeta) ds(\zeta) + \int_\D |\xi|^{2n} g_z(\xi)\, dA(\xi) \\
    &= \int_{\partial\D} \frac{\partial_n g_z}{4(n+1)^2}\, ds + \int_\D |\xi|^{2n} g_z(\xi)\, dA(\xi) \\
    &= \frac{1}{4(n+1)^2} + \int_\D |\xi|^{2n} g_z(\xi)\, dA(\xi),
  \end{align*}
  where we've used the facts that $|\zeta|=1$ on $\partial \D$ and that $\partial_n g_z\, ds$ is a
  probability measure on $\partial \D$. We conclude that
  \begin{equation*}
    \int_\D |\xi|^{2n} g_z(\xi)\, dA(\xi) = -\frac{1-|z|^{2n+2}}{4(n+1)^2},
  \end{equation*}
  as desired.
\end{proof}
\begin{example}
  For a simple example, consider the Dirichlet problem for the Helmholtz equation on the unit disk $\D$:
  \begin{equation*}
  \begin{cases}
    (\Delta - a) \phi_a = 0 & \textrm{ in } \D \\
    \phi_a = 1 &\textrm{ on } \partial \D,
  \end{cases}
  \end{equation*}
  where $0 < a \ll 1$. Clearly $\phi_0\equiv 1$, so we obtain
  \begin{equation} \label{Dirichlet_example}
    \phi_a(z) = 1 + a\int_\D g_z\, dA + o(a),
  \end{equation}
  where the Green function $g_z$ is given by
  \begin{equation*}
    g_z(\xi) = \frac{1}{2\pi}\ln\left|\frac{\xi-z}{1-\overline{z}\xi}\right|.
  \end{equation*}
  Lemma \ref{Green_area} gives
  \begin{equation*}
    \int_\D g_z\, dA(\xi)
      = -\frac{1-|z|^2}{4},
  \end{equation*}
  so equation \eqref{Dirichlet_example} becomes
  \begin{equation*}
    \phi_a(z) = 1 - \frac{a}{4}(1-|z|^2) + o(a).
  \end{equation*}
\end{example}

\subsection{The Dirichlet problem for Laplace--Beltrami}
\begin{theorem}
  Let $D\subset \C$ be a bounded domain with smooth, analytic boundary and suppose $u\in C^\infty(\overline{D})$ is a positive function.
  For $\epsilon \geq 0$ define the Laplace--Beltrami operator $\lambda = 1 + \epsilon u$.
  Given a bounded Borel function $f$ the solution $\phi_\epsilon$ to the Dirichlet problem
  \begin{equation*}
    \begin{cases}
      \nabla\cdot(\lambda \nabla \phi_\epsilon) = 0 & \textrm{ in } D\\
      \phi_\epsilon = f &\textrm{ on } \partial D
    \end{cases}
  \end{equation*}
  satisfies
  \begin{equation*}
    \phi_\epsilon(z) = \phi_0(z) + \epsilon \int_D g_z \nabla u \cdot \nabla \phi_0 \, dA + o(\epsilon)
  \end{equation*}
  as $\epsilon \to 0$, where the error term converges uniformly in $z$.
\end{theorem}
\begin{proof}
  First note that within $D$,
  \begin{equation*}
    \nabla\cdot(\lambda\nabla\phi_0) = \lambda \Delta \phi_0 + \nabla\lambda \cdot\nabla\phi_0
    = \epsilon\nabla u \cdot\nabla \phi_0.
  \end{equation*}
  Therefore
  \begin{equation*}
    \begin{cases}
      \nabla\cdot[\lambda \nabla (\phi_\epsilon-\phi_0)] = \epsilon\nabla u\cdot\nabla \phi_0 & \textrm{ in } D\\
      \phi_\epsilon-\phi_0 = 0 &\textrm{ on } \partial D.
    \end{cases}
  \end{equation*}
  This problem can be solved by integrating against $g^*$, the Green function of $\nabla\lambda\nabla$:
  \begin{equation*}
    \phi_\epsilon(z) - \phi_0(z) = \epsilon\int_D g^*_z \nabla u\cdot \nabla \phi_0\, dA.
  \end{equation*}
  Using the perturbation formula \eqref{Beltrami1} gives
  \begin{equation*}
    \phi_\epsilon(z) = \phi_0(z) + \epsilon\int_D g_z \nabla u\cdot \nabla \phi_0\, dA + o(\epsilon). \qedhere
  \end{equation*}
\end{proof}
\begin{example}
  Returning to the unit disk, let $u(\xi) = |\xi|^2$ and $f(x,y) = x^2-y^2$ on $\partial \D$.
  Given $\epsilon \ll 1$, consider the Dirichlet problem
  \begin{equation*}
    \begin{cases}
      \nabla\cdot((1+\epsilon u) \nabla \phi_\epsilon) = 0 & \textrm{ in } \D\\
      \phi_\epsilon = f &\textrm{ on } \partial \D.
    \end{cases}
  \end{equation*}
  Notice that $\phi_0(x,y) = x^2-y^2$; that is, $\phi_0(z) = z^2+\overline{z}^2$.
  We compute
  \begin{equation*}
    \nabla u(\xi) \cdot\nabla \phi_0(\xi) = \Re(2\xi\cdot 2\overline{\xi}) = 4|\xi|^2,
  \end{equation*}
  so the first variation of the solution is given by
  \begin{equation*}
    \delta\phi_\epsilon(z) = \int_\D 4|\xi|^2 g_z(\xi)\, dA(\xi).
  \end{equation*}
  Lemma \ref{Green_area} gives
  \begin{equation*}
    \delta\phi_\epsilon(z) = -\frac{1-|z|^{4}}{4},
  \end{equation*}
  so we have
  \begin{equation*}
    \phi_\epsilon(x,y) = x^2-y^2 - \frac{\epsilon}{4}(1-x^2-y^2) + o(\epsilon).
  \end{equation*}
\end{example}

\section{Further Characterization of Elliptic Growth Processes}
In the introduction we discussed a few successes and shortcomings of Laplacian growth as a model for fluid flow in a porous medium
or in a Hele--Shaw cell. We now turn to elliptic growth, one of the aforementioned generalizations of Laplacian growth
that takes permeability to be a nonconstant scalar function. Recall the dynamics: we say that a family of domains $\{D(t): 0\leq t<t_0\}$
undergoes elliptic growth with permeability $\lambda$ and flow rate $Q$ if for some $w\in\cap_t D(t)$ we have
  \begin{equation}\label{EG_definition}
    \begin{cases}
      \nabla\lambda\nabla g = Q\delta_w & \textrm{ in } D\\
      g=0 & \textrm{ on } \partial D \\
      v_n = \lambda\partial_n g & \textrm{ on } \partial D,
    \end{cases}
  \end{equation}
where $g$ denotes the Green function of $D(t)$ for the operator $\nabla\lambda\nabla$ and $v_n$ denotes the outward normal velocity
of the boundary $\partial D(t)$. Elliptic growth in this form was first described in
\cite{Khavinson}, wherein the authors also describe a type of elliptic growth replacing $\nabla\lambda\nabla$ with a Schr\"odinger
operator $\Delta - u$. In this section we will describe a few basic properties of elliptic growth of both types.
The first result is an analogue of Richardson's theorem; this theorem and its corollary appeared previously in \cite{Khavinson}.

\begin{theorem}
  Let $D(t)\subset\C$ be a growing family of domains with $C^2$ boundaries described by an elliptic growth process
  via a Laplace--Beltrami operator $L=\nabla\lambda\nabla$ and singularity $w$. If $\phi$ is a smooth function satisfying $L\phi=0$, then
  \begin{equation*}
    \frac{d}{dt}\int_{D(t))} \phi\, dA = \phi(w).
  \end{equation*}
\end{theorem}
\begin{proof}
  Call the relevant Green function $g_w$. The outward normal velocity of the boundary is $\lambda\partial_n g_w$,
  so we have
  \begin{align*}
    \frac{d}{dt}\int_{D(t))} \phi\, dA &= \int_{\partial D(t)} \phi\lambda\partial_ng_w\, ds
    = \int_{\partial D(t)} \lambda(\phi\partial_ng_w - g_w\partial_n \phi)\, ds \\
    &= \int_{D(t)} (\phi Lg_w - g_wL\phi)\, dA
    = \phi(w),
  \end{align*}
  where we've use the facts that $L\phi = 0$ in $D(t)$, $g_w = 0$ on $\partial D(t)$, and $Lg_w = \delta_w$ in $D(t)$.
\end{proof}
\begin{corollary}\label{LB_area_growth}
  With the assumptions of the previous theorem,
  \begin{equation*}
    \frac{d}{dt}\mathrm{area}(D(t)) = 1.
  \end{equation*}
\end{corollary}
\begin{proof}
  Take $\phi\equiv 1$ in the previous theorem.
\end{proof}

\begin{lemma} \label{slow_growth}
  Let $D(t)\subset\C$ be a growing family of domains with $C^2$ boundaries described by an elliptic growth process
  via a Schr\"odinger operator. Then at each time $t$,
  \begin{equation*}
    \frac{d}{dt}\mathrm{area}(D(t)) \leq 1.
  \end{equation*}
\end{lemma}
\begin{proof}
  Calling the singularity $w$ and the operator $\Delta-u$, we have
  \begin{equation*}
    \frac{d}{dt}\mathrm{area}(D(t)) = \int_{\partial D(t)} \partial_n g_w\, ds = \int_{D(t)} \Delta g_w\, dA.
  \end{equation*}
  Using the fact that $(\Delta-u)g_w = \delta_w$ gives
  \begin{equation*}
    \frac{d}{dt}\mathrm{area}(D(t))
    = \int_{D(t)} (\delta_w +ug_w)\, dA
    = 1 + \int_{D(t)} ug_w\, dA.
  \end{equation*}
  Since $u\geq 0$ and $g_w\leq 0$, the result follows.
\end{proof}
Together these three results show that elliptic growth of Laplace--Beltrami type features the same constant area increase
of Laplacian growth, whereas elliptic growth of Schr\"odinger type is slower. We can refine this idea a bit and show that
Schr\"odinger--type growth initially has the same rate of area change as the other growth processes, assuming that
the initial domain has no spatial extent (that is, we begin pumping fluid into an empty medium).
\begin{lemma} \label{positive_lemma}
  Let $D\subset \C$ be a bounded domain with $C^2$ boundary.
  Given a nonnegative function $u\in C^\infty(D)$, there exists a function
  $\varphi\in C^\infty(D)\cap C(\overline{D})$ so that $(\Delta-u)\varphi = 0$ and $\varphi > 0$ in $D$.
\end{lemma}
\begin{proof}
  The Dirichlet problem
  \begin{equation*}
    \begin{cases}
      (\Delta -u) \varphi = 0 & \textrm{ in } D \\
      \varphi=1 & \textrm{ on } \partial D
    \end{cases}
  \end{equation*}
  has a unique solution which is continuous on $\overline{D}$. By compactness of $\overline{D}$,
  the solution $\varphi$ must attain a minimum value. If this
  minimum is on $\partial D$, then $\varphi > 0$ everywhere. If the minimum occurs at $z\in D$, then by Hopf's maximum principle (theorem
  3.5 in \cite{Gilbarg}) either $\varphi(z) > 0$ or $\varphi$ is constant. In the former case, $\varphi > 0$ throughout $D$;
  in the latter case $\varphi \equiv 1$ and once again is positive everywhere.
\end{proof}
\begin{theorem}
  Let $\{D(t): 0 < t < t_0\}$ be a growing family of domains with $C^2$ boundaries produced by an elliptic growth process
  of Schr\"odinger type with singularity at $w$.
  If we assume that
  \begin{equation*}
    \lim_{t\to 0} \sup_{\zeta\in \partial D(t)} |\zeta-w| = 0,
  \end{equation*}
  then
  \begin{equation*}
    \lim_{t\to 0} \frac{d}{dt}\mathrm{area}(D(t)) = 1.
  \end{equation*}
\end{theorem}
\begin{proof}
  From lemma \ref{slow_growth} we know that the family of growing domains is bounded, so let $D$ be a large disk containing
  each $D(t)$. Denote the underlying operator by $\Delta-u$ for some $u\in C^\infty(D)$ and the singularity
  of the growth process by $w$.
  From the proof of theorem \ref{slow_growth}, we see it suffices to show that
  \begin{equation*}
    \lim_{t \to 0}\int_{D(t)} uG_w\, dA = 0,
  \end{equation*}
  where $G_w$ denotes the Green function of $\Delta-u$ on $D(t)$. Using lemmas
  \ref{positive_lemma} and \ref{conversion_lemma} we can find a positive
  smooth function $\lambda$ so that $u=\lambda^{-1/2}\Delta\lambda^{1/2}$ and
  \begin{equation*}
    G_w(z) = g_w(z)\sqrt{\lambda(w)\cdot\lambda(z)},
  \end{equation*}
  where $g_w$ denotes the Green function of $\nabla\lambda\nabla$. Thus it suffices to show that
  \begin{equation*}
    \lim_{t\to 0}\int_{D(t)} g_w\Delta\sqrt{\lambda}\, dA = 0.
  \end{equation*}
  We proceed by writing this integral in two ways. Since $g_w=0$ on $\partial D(t)$, two applications of integration by parts gives
  \begin{align}
    \int_{D(t)} g_w\Delta\sqrt{\lambda}\, dA &= -\int_{D(t)} \nabla\sqrt{\lambda} \cdot \nabla g_w\, dA \notag\\
    &= \int_{D(t)} \sqrt{\lambda}\Delta g_w \, dA - \int_{\partial D(t)} \sqrt{\lambda}\partial_n g_w\, ds. \label{area_limit_equation}
  \end{align}
  Alternatively, we can write
  \begin{align*}
    \int_{D(t)} g_w\Delta\sqrt{\lambda}\, dA &= -\int_{D(t)} \nabla\sqrt{\lambda} \cdot \nabla g_w\, dA \\
    &= -\int_{D(t)} \frac{\nabla\lambda \cdot \nabla g_w}{2\sqrt{\lambda}}\, dA \\
    &= -\int_{D(t)} \frac{\delta_w - \lambda\Delta g_w}{2\sqrt{\lambda}}\, dA \\
    &= \frac{1}{2}\int_{D(t)} \sqrt{\lambda}\Delta g_w\, dA - \frac{1}{2\sqrt{\lambda(w)}}.
  \end{align*}
  Combining this with equation \eqref{area_limit_equation} gives
  \begin{equation*}
    \int_{D(t)} \sqrt{\lambda}\Delta g_w \, dA = 2\int_{\partial D(t)} \sqrt{\lambda}\partial_n g_w\, ds
      - \frac{1}{\sqrt{\lambda(w)}}.
  \end{equation*}
  Inserting this back into \eqref{area_limit_equation} gives
  \begin{equation*}
    \int_{D(t)} g_w\Delta\sqrt{\lambda}\, dA = \int_{\partial D(t)} \sqrt{\lambda}\partial_n g_w\, ds - \frac{1}{\sqrt{\lambda(w)}}.
  \end{equation*}
  Finally, since $\lambda$ is bounded away from 0 and $\lambda\partial_n g_w\, ds$ is a probability measure on $\partial D(t)$,
  we have
  \begin{align*}
    \left|\int_{D(t)} g_w\Delta\sqrt{\lambda}\, dA\right|
      &= \left|\int_{\partial D(t)} \left(\frac{1}{\sqrt{\lambda}} - \frac{1}{\sqrt{\lambda(w)}}\right)\lambda\partial_n g_w\, ds\right| \\
      &\leq \left\|\frac{1}{\sqrt{\lambda}} - \frac{1}{\sqrt{\lambda(w)}}\right\|_{\infty, \overline{D(t)}}
        \int_{\partial D(t)} \lambda \partial_n g_w\, ds \\
      &= \left\|\frac{1}{\sqrt{\lambda}} - \frac{1}{\sqrt{\lambda(w)}}\right\|_{\infty, \overline{D(t)}}.
  \end{align*}
  As $\lambda^{-1/2}$ is continuous on the large disk $D$, it is uniformly continuous near $w$. Our assumption
  on the (Hausdorff) convergence of $\partial D(t)$ to the point $w$ is enough to conclude that
  \begin{equation*}
    \left|\int_{D(t)} g_w\Delta\sqrt{\lambda}\, dA\right|
      \leq \left\|\frac{1}{\sqrt{\lambda}} - \frac{1}{\sqrt{\lambda(w)}}\right\|_{\infty, \overline{D(t)}} \to 0,
  \end{equation*}
  as desired.
\end{proof}
Our next problem in studying elliptic growth of Schr\"odinger type is the difficulty one has in computing simple examples. For instance,
it is clear that a radially--symmetric potential will grow disks into larger disks, but at what rate? The following theorem
allows us to produce a class of examples.
\begin{theorem} \label{example_maker}
  For $R>0$ let $D_R\subset \C$ denote the disk of radius $R$ centered at $0$. For a smooth, nondecreasing radial function
  $\lambda > 0$ define the potential
  \begin{equation*}
    u = \frac{\Delta\sqrt{\lambda}}{\sqrt{\lambda}},
  \end{equation*}
  also a radial function. If $D_R$ undergoes elliptic growth with operator $\Delta-u$ and singularity at the origin, then we produce
  an increasing family of disks with areas changing at the rate
  \begin{equation*}
    \frac{d}{dt}\mathrm{area}(D_R) = \frac{\sqrt{\lambda(0)}}{\sqrt{\lambda(R)}}.
  \end{equation*}
\end{theorem}
\begin{proof}
  First we should verify that $u\geq 0$; since $\lambda$ is nondecreasing we have for any $R > 0$
  \begin{equation*}
    \sqrt{\lambda(0)} \leq \sqrt{\lambda(R)} = \frac{1}{2\pi R} \int_{\partial D_R} \sqrt{\lambda} \, ds,
  \end{equation*}
  where we have used the fact that $\lambda$ is constant on any circle $\partial D_R$. We conclude that $\sqrt{\lambda}$ is subharmonic,
  so $u\geq 0$. Note that this argument shows the hypothesis that $\lambda$ is nondecreasing is necessary to have $u\geq 0$.
  
  The Green function $G_0$ of $\Delta - u$ on $D_R$ with singularity at 0 must be radial by the rotational symmetry of $\Delta,u$,
  and $D_R$. Thus $\partial_n G_0$ is constant on $\partial D_R$. If elliptic growth of $D_R$ occurs via the operator $\Delta-u$
  with singularity at 0, then $\partial D_R$ moves outward in all directions at an equal rate. Therefore we produce an increasing
  family of disks.
  
  If $g_0$ denotes the Green function of $\nabla\lambda\nabla$, then lemma \ref{conversion_lemma} yields
  \begin{equation*}
    G_0(z) = g_0(z)\sqrt{\lambda(0)\cdot \lambda(z)}.
  \end{equation*}
  Note that $g_0$ is also radial.
  For $\zeta\in\partial D_R$ we have
  \begin{equation*}
    \partial_n G_0(\zeta) = g_0(\zeta)\partial_n\sqrt{\lambda(0)\cdot \lambda(\zeta)} + \partial_n g_0(\zeta)\sqrt{\lambda(0)\cdot
      \lambda(\zeta)}
    = \partial_n g_0(\zeta)\sqrt{\lambda(0)\cdot \lambda(\zeta)},
  \end{equation*}
  where we have used the fact that $g_0 = 0$ on $\partial D_R$. As $\lambda$ and $\partial_n g_0$ are constant on $\partial D_R$,
  we have
  \begin{align*}
    2\pi R \lambda \partial_n g_0
    &= \int_{\partial D_R} \lambda \partial_n g_0 \, ds
    = \int_{\partial D_R} \left(\lambda \partial_n g_0 +g_0\partial_n\lambda\right)\, ds \\
    &= \int_{\partial D_R} \partial_n \left(\lambda g_0\right)\, ds
    = \int_{D_R} \nabla\lambda\nabla g_0\, dA \\
    &= 1.
  \end{align*}
  Thus $\lambda \partial_n g_0 = (2\pi R)^{-1}$, so we can write
  \begin{equation*}
    \partial_n G_0(\zeta) = \frac{\sqrt{\lambda(0)}}{\sqrt{\lambda(\zeta)}}\cdot \lambda(\zeta) \partial_n g_0(\zeta)
      = \frac{\sqrt{\lambda(0)}}{2\pi R\sqrt{\lambda(R)}}.
  \end{equation*}
  Finally we have
  \begin{equation*}
    \frac{d}{dt} \mathrm{area}(D_R) = 2\pi R \cdot \frac{dR}{dt} = 2\pi R \cdot \partial_n G_0 = \frac{\sqrt{\lambda(0)}}{\sqrt{\lambda(R)}},
  \end{equation*}
  as desired.
\end{proof}
\begin{corollary}
  Suppose a disk of radius $R>0$ undergoes elliptic growth via the Helmholtz operator $\Delta - 1$ with singularity at the center.
  The area changes at the rate
  \begin{equation*}
    \frac{d}{dt} \mathrm{area}(D_R) = \frac{1}{I_0(R)},
  \end{equation*}
  where $I_0$ denotes the zeroth--order modified Bessel function of the first kind.
\end{corollary}
\begin{proof}
  Using translational symmetry of the Helmholtz operator, let's assume the center of the disk is the origin.
  We can use theorem \ref{example_maker} if we find a radial function $\lambda=\lambda(r)$ so that
  \begin{equation*}
    \frac{\Delta\sqrt{\lambda}}{\sqrt{\lambda}} = 1.
  \end{equation*}
  Writing the Laplacian in polar coordinates gives
  \begin{equation*}
    \left(\partial_r^2 + r^{-1} \partial_r \right)\sqrt{\lambda} = \sqrt{\lambda},
  \end{equation*}
  which rearranges into
  \begin{equation*}
    \left(r^2\partial_r^2 + r \partial_r -r^2 \right)\sqrt{\lambda} = 0.
  \end{equation*}
  This is the modified Bessel equation of order 0; as $\lambda$ must be bounded near 0, we choose the solution
  \begin{equation*}
    \sqrt{\lambda(r)} = I_0(r).
  \end{equation*}
  Since $I_0(0) = 1$, the result follows from theorem \ref{example_maker}.
\end{proof}
\begin{corollary}
  Suppose a disk of radius $R>0$ centered at the origin undergoes elliptic growth via the Schr\"odinger
  operator $\Delta - 4(r^2+1)$ with singularity at the center.
  The area changes at the rate
  \begin{equation*}
    \frac{d}{dt} \mathrm{area}(D_R) = \exp(-R^2).
  \end{equation*}
\end{corollary}
\begin{proof}
  Let $\lambda(r) = \exp(2r^2)$, from which we compute
  \begin{equation*}
    \frac{\Delta\sqrt{\lambda}}{\sqrt{\lambda}} = \frac{(\partial^2_r + r^{-1}\partial_r)\exp(r^2)}{\exp(r^2)} = 4(r^2+1).
  \end{equation*}
  Since $\lambda(0) = 1$, the result follows from theorem \ref{example_maker}.
\end{proof}

\section{Inverse Problems}
\subsection{A Local Inverse Problem}
As a generalization of Laplacian growth, elliptic growth should produce a wider range of phenomena.
For example Laplacian growth always preserves disks, but a non--radial permeability function can cause a growing
disk to break its rotational symmetry. Thus we are led to the following question:
Which families of nested, growing domains can be produced by an elliptic growth process?
Our first approach to addressing this question is to consider the local problem for growth
of Schr\"odinger type. Theorem \ref{NormalTheorem} tells us that
\begin{equation*}
  \delta\partial_n g^*_w(\zeta) = \int_D ug_wP_\zeta\, dA,
\end{equation*}
so we are led to the following result, which is proven in \cite{Martin2}.
\begin{theorem}
  Let $D \subset \C$ be a bounded domain with smooth, analytic boundary and fix $w\in D$.
  Define
  \begin{equation*}
     Au(\zeta) = \int_D u(z)g(z,w)P(z,\zeta)\, dA(z).
  \end{equation*}
  Then $A$ is a linear map $L^2(D)\to L^2(\partial D)$ with dense range.
\end{theorem}
As we preturb the Laplacian into nearby Schr\"odinger operators $\Delta - \epsilon u$, the function $Au$ gives the first order
corrections to the velocity profile at the boundary. Hence locally, we can get ``almost all'' velocity profile corrections
via an appropriate choice of $u$ and $\epsilon$.

There is another, unsolved question related to injectivity of this map from operators to boundary velocities.
It is known that elliptic growth for various operators can produce the same growth process---for instance a radial permeability
function $\lambda$ will preserve disks in elliptic growth of Laplace--Beltrami type, while corollary \ref{LB_area_growth}
states that the growth occurs at a rate independent of $\lambda$. Hence any radial lambda produces the same growth
dynamics for a disk. It is unknown if there are other sources of noninjectivity, and in particular none are known for
Schr\"odinger--type growth. We are led to the following conjecture.
\begin{conjecture}
  If $D(t)$ is described by two elliptic growth processes with respect to operators $\Delta - u_1$ and $\Delta - u_2$,
  then $u_1=u_2$. Restated into purely function--theoretic terms, if $g_w$ and $g_w^*$ are Green functions of
  $\Delta-u_1$ and $\Delta-u_2$ respectively with the same singularity and $\partial_ng_w = \partial_n g_w^*$
  everywhere on $\partial D$, then $u_1=u_2$.
\end{conjecture}

Next we prove a negative result. For a growing family of domains $\{D(t):t\geq 0\}$ to be an elliptic growth process driven by a
Laplace--Beltrami operator, the time derivative of the area of $D(t)$ must be 1. Generally this is not a major constraint,
since an appropriate time rescaling can force this condition. That is, unless the time derivative of area vanishes. This simple
observation yields the following theorem.
\begin{theorem}
  If $\{D(t):t\geq 0\}$ is a growing family of smooth, analytic domains such that
  \begin{equation*}
    \frac{d}{dt}\mathrm{area}(D(t))\bigg|_{t=t_0} = 0
  \end{equation*}
  at some time $t_0$, then no elliptic growth of Laplace--Beltrami type can drive the dynamics.
\end{theorem}
\begin{example}
  Consider the unit disk $\D\subset \C$ as an ellipse with both foci at the center. We can produce a growing family of domains
  by separating the foci at a constant rate so that the center of the ellipse stays fixed; furthermore,
  we can preserve the length of the
  semiminor axis. For concreteness, say the foci are located at $\pm c(t)$ on the real axis, with $c(0) = 0$.
  The semimajor axis $b$ must change with $c$, via the relation
  \begin{equation*}
    2b\dot{b} = \frac{d}{dt}b^2= \frac{d}{dt}\left(a^2+c^2\right) = 2c\dot{c}.
  \end{equation*}
  Then the area changes at the rate
  \begin{equation*}
    \dot{A} = \frac{d}{dt}(\pi a b) = \frac{\pi a c \dot{c}}{b}.
  \end{equation*}
  But at $t=0$ we have $\dot{A}=0$, so no elliptic growth process of Laplace--Beltrami type can describe this family of domains.
\end{example}
With more work we can prove a similar theorem for growth of Schr\"odinger type. Even though the rate of area increase
is generally less than 1, this theorem shows it can never be 0.
\begin{theorem}
  If $\{D(t):t\geq 0\}$ is a growing family of smooth, analytic domains such that
  \begin{equation*}
    \frac{d}{dt}\mathrm{area}(D(t))\bigg|_{t=t_0} = 0
  \end{equation*}
  at some time $t_0$, then no elliptic growth of Schr\"odinger type can drive the dynamics.
\end{theorem}
\begin{proof}
  Assume for the sake of contradiction that the family of domains $\{D(t):t\geq 0\}$ are generated by elliptic growth
  via the operator $\Delta - u$ and that at some time $t_0$ the time derivative of area vanishes. Since $\partial_n g_w \geq 0$
  along $\partial D(t)$ and
  \begin{equation*}
    0 = \frac{d}{dt}\int_{D(t)} dA\bigg|_{t=t_0} = \int_{\partial D(t_0)} \partial_n g_w\, ds,
  \end{equation*}
  we must have that $\partial_n g_w = 0$ everywhere on $\partial D(t_0)$.
  Using lemma \ref{positive_lemma}, we can find a smooth, positive function $\lambda$ so that $(\Delta- u)\sqrt{\lambda} = 0$. That is,
  \begin{equation*}
    u = \frac{\Delta \sqrt{\lambda}}{\sqrt{\lambda}}
  \end{equation*}
  throughout $D$. We proceed as in the proof of theorem \ref{example_maker}; let $G_w$ and $g_w$ denote
  the Green functions of $\Delta- u$ and $\nabla \lambda\nabla$, respectively. By lemma \ref{conversion_lemma},
  the Green functions can be related:
  \begin{equation*}
    G_w(z) = g_w(z) \sqrt{\lambda(w)\cdot\lambda(z)}.
  \end{equation*}
  Since $g_w = \partial_n g_w = 0$ on $\partial D(t_0)$, we deduce
  that $\partial_n G_w = 0$ everywhere on $\partial D(t_0)$. The divergence theorem yields
  \begin{equation*}
    0 = \int_{\partial D(t_0)} \lambda \partial_n G_w\, ds = \int_{D(t_0)} \nabla\lambda\nabla G_w\, dA = 1,
  \end{equation*}
  a contradiction.
\end{proof}

\subsection{Relation to the Calder\'on Problem}
Consider a conducting body $D\subseteq \C$ with a nonconstant isotropic conductivity function $\lambda\in C^2(\overline{D})$. If $D$
is made of a nonhomogeneous material, $\lambda$ will reflect this; conversely, knowledge of $\lambda$ indicates inhomogeneities
within the material. If $D$ is an object which we do not wish to deconstruct---such as a living creature,
as in the case of medical imaging---we might wish to determine $\lambda$ from measurements made at the boundary of $D$.
If we induce a voltage potential $f\in H^{1/2}(\partial D)$, then the power required to maintain the potential is
\begin{equation*}
  Q_\lambda(f) = \int_{\partial D} f\lambda\partial_n u\, ds,
\end{equation*}
where $u$ solves the Dirichlet problem $\nabla\lambda\nabla u = 0$ in $D$ and $u=f$ on $\partial D$.
The functional $Q_\lambda$ is readily measured, so we might ask if we can construct $\lambda$ from knowledge of $Q_\lambda$.
This problem was originally stated by Calder\'on in \cite{Calderon} and is now known as the Calder\'on problem.

Another way of stating the problem involves the so--called Dirichlet--to--Neumann map $N_\lambda: H^{1/2}(\partial D) \to
H^{-1/2}(\partial D)$,
which maps $f\mapsto \partial_n u$, where $u$ solves the aforementioned Dirichlet problem.
For an induced voltage potential $f$, the function $N_\lambda f$ is the resulting electric field at the boundary.
Since the field can be measured, can we construct $\lambda$ from knowledge of $N_\lambda$?
Much work has been done on this problem; for example, in \cite{Calderon_soln} the authors show that a solution to the problem
is unique for $\lambda$ smooth, while in \cite{Nachman} the author discusses constructive methods to determining $\lambda$.
We will now describe a different inverse problem related to elliptic growth which can be reduced to the same
question.

Suppose we have a fluid domain $D\subseteq \C$ occupying a region with a nonconstant, isotropic permeability function
$\lambda \in C^\infty(\overline{D})$.
If $\lambda$ is unknown, we might have a way of determining it from observing the movement of the fluid;
that is, if we pump an infinitesimal amount of fluid at a point $p\in D$, we can observe the resulting fluid velocity at the boundary.
Are these observations enough to construct $\lambda$?

If $g$ denotes the Green function of $D$ for the operator $\nabla\lambda\nabla$, then pumping at a point $p\in D$
yields a boundary velocity profile $V(p) = \lambda\partial_n g_p$,
so we have knowledge of the map $V: p\mapsto \lambda \partial_n g_p$.
Notice that $\lambda\partial_n g$ is the Poisson kernel for $\nabla\lambda\nabla$, and as such
\begin{equation*}
  u(p) = \int_{\partial D} f V(p)\, ds
\end{equation*}
solves the Dirichlet problem $\nabla\lambda\nabla u = 0$ in $D$ and $u=f$ on $\partial D$. From here we can construct
the Dirichlet--to--Neumann map $N_\lambda f = \partial_n u$, reducing the problem to that of Calder\'on.
This is not the only way of approaching this inverse problem, but it is a nice connection to a known topic.
To summarize, we give the following theorem.
\begin{theorem}
  Suppose we have a bounded $C^2$ domain $D\subseteq \C$ with an unknown nonconstant isotropic permeability function
  $\lambda \in C^\infty(\overline{D})$. Assume we have knowledge of a pumping response function $\phi: D\to C(\partial D)$
  which maps a point $p\in D$ to the resulting velocity profile when fluid is pumped into $D$ at $p$; that is,
  $\phi(p) = \lambda\partial_n g_p$. Then the problem of finding $\lambda$ from $\phi$ reduces to that of Calder\'on (and hence
  is solvable).
\end{theorem}
As a final remark in this brief section, we note that there is much interest in the Calder\'on problem with only partial data---such
as measurements on only fractions of the boundary (for example, see \cite{Nachman2}).
We already know that the pumping response of a domain at a single point
cannot detect the permeability (again due to the example of all radial permeabilities growing disks in the same way), but
what about knowldege of the pumping response function at two points? Or at a finite number of points?
\begin{conjecture}
  If both $\lambda\partial_n g_w$ and $\lambda\partial_n g_\xi$ are known for distinct points $w,\xi\in D$ then $\lambda$
  can be determined.
\end{conjecture}

\section{Elliptic Growth as Balayage}

In the introduction we gave a brief discussion of balayage and its relationship to Laplacian growth. Specifically,
if $D(0)$ evolves into $D(t)$ during Laplacian growth after time $t$, then
the balayage of $\chi_{D(0)} + t\delta_w$ with respect to Lebesgue measure and $\chi_{D(t)}$ coincide up to null sets.
In this section we will generalize to balayage of measures based upon an underlying Laplace--Beltrami operator
and show that elliptic growth follows as Laplacian growth did before.
Since no boundary regularity is needed,
this extends the definition of elliptic growth to domains possessing cusps, corners, or even fractal boundaries.

Throughout this section, let $\lambda:\C\to\R$ be a smooth function satisfying $\lambda\geq \lambda_0 > 0$ everywhere for some constant
$\lambda_0$ and define $L=\nabla\lambda\nabla$. Suppose further that $\mu$ is a positive and finite Borel measure on $\C$;
ultimately we only wish to consider measures of the form $\mu = \chi_D + t\delta_w$. Finally, fix a fundamental solution
$E$ of $L$ and define the elliptic potential of $\mu$ as
\begin{equation*}
  \Lambda^\mu(z) = \int E(z-w)\, d\mu(w),
\end{equation*}
so that $L\Lambda^\mu = \mu$.
We define the (elliptic) balayage
of $\mu$ with respect to Lebesgue measure as follows. Find the smallest function $V$ satisfying both
\begin{equation}\label{elliptic_balayage_definition}
  \begin{cases}
    V \geq \Lambda^\mu \\
    L V \leq 1
  \end{cases}
\end{equation}
throughout $\C$; the balayage of $\mu$ is then $\Bal_\lambda(\mu,1) = LV$. Before proceeding, we should
ensure this definition is sound.
\begin{theorem}
  There exists a unique smallest function $V$ satisfying \eqref{elliptic_balayage_definition}.
\end{theorem}
\begin{proof}
  Find a smooth function $u$ so that $Lu=1$. Then we wish to find the smallest $V$ satisfying both
  \begin{equation*}
    \begin{cases}
      V-u \geq \Lambda^\mu-u \\
      L (V-u) \leq 0
    \end{cases}
  \end{equation*}
  throughout $\C$. This is the elliptic version of a standard obstacle problem: find a smallest superharmonic function which is bounded
  below by a known function. The existence of a unique solution is well--known (e.g., \cite{Obstacle}).
\end{proof}
\noindent
Next we want to verify that the balayage of $\chi_D+t\delta_w$ is the characteristic function of a domain in this generalized setting.
We will need the following two lemmas.
\begin{lemma}
  Given positive and finite Borel measures $\mu_1,\mu_2$, we have
  \begin{equation*}
    \Bal_\lambda(\mu_1+\mu_2,1) = \Bal_\lambda(\Bal_\lambda(\mu_1,1)+\mu_2,1).
  \end{equation*}
\end{lemma}
\begin{proof}
  Given a measure $\mu$ let $V^\mu$ denote the smallest function satisfying \eqref{elliptic_balayage_definition} and let
  $F(\mu) = \Bal_\lambda(\mu,1) = LV^\mu$. With this notation we wish to show $V^{F(\mu_1)+\mu_2} = V^{\mu_1+\mu_2}$. Note that
  \begin{equation*}
    V^{F(\mu_1)+\mu_2} \geq \Lambda^{F(\mu_1)+\mu_2} = \Lambda^{F(\mu_1)} + \Lambda^{\mu_2} = V^{\mu_1} + \Lambda^{\mu_2}
    \geq \Lambda^{\mu_1} + \Lambda^{\mu_2} = \Lambda^{\mu_1+\mu_2}.
  \end{equation*}
  Since $LV^{F(\mu_1)+\mu_2} \leq 1$, considering the minimization problem of $\Bal_\lambda(\mu_1+\mu_2,1)$ gives
  $V^{F(\mu_1)+\mu_2} \geq V^{\mu_1+\mu_2}$. Next, note that
  \begin{equation*}
    V^{\mu_1+\mu_2} - \Lambda^{\mu_2} \geq \Lambda^{\mu_1+\mu_2} - \Lambda^{\mu_2} = \Lambda^{\mu_1}
  \end{equation*}
  and that the positivity of $\mu_2$ gives
  \begin{equation*}
    L(V^{\mu_1+\mu_2} - \Lambda^{\mu_2}) \leq 1-\mu_2 \leq 1.
  \end{equation*}
  Considering the minimization problem for $\Bal_\lambda(\mu_1,1)$ gives
  \begin{equation*}
    V^{\mu_1+\mu_2} - \Lambda^{\mu_2} \geq V^{\mu_1} = \Lambda^{F(\mu_1)},
  \end{equation*}
  whence $V^{\mu_1+\mu_2} \geq \Lambda^{F(\mu_1)+\mu_2}$. As $LV^{\mu_1+\mu_2} \leq 1$, the considering the minimization problem
  for $\Bal_\lambda(F(\mu_1)+\mu_2,1)$ gives $V^{\mu_1+\mu_2} \geq V^{F(\mu_1)+\mu_2}$, as desired.
\end{proof}
\begin{lemma}
  Let $\mu$ be a measure that has either the form $t\delta_w$ for some $t > 0$ and $w\in\C$ or
  the form $\mu = f\, dm$ for some $f\in L^\infty(dm)$. Then there exists a domain $D$ so that
  \begin{equation*}
    \Bal_\lambda(\mu,1) = \chi_{D}.
  \end{equation*}
\end{lemma}
\begin{proof}
  A proof of this result has appeared in a few different forms for partial balayage based upon the Laplacian; see
  \cite{Balayage,Balayage2,Sakai} for details. The argument for elliptic balayage is analogous.
\end{proof}
\begin{theorem}
  Given a bounded domain $D(0)\subset \C$, a point $w\in D$, and $t\geq 0$, there exists a domain $D(t)\supseteq D(0)$ so that
  \begin{equation*}
    \Bal_\lambda(\chi_{D(0)} + t\delta_w,1) = \chi_{D(t)}.
  \end{equation*}
\end{theorem}
\begin{proof}
  Using the lemmas above, there exist domains $D'(t)$ and $D(t)$ so that
  \begin{align*}
    \Bal_\lambda(\chi_{D(0)} + t\delta_w,1) &= \Bal_\lambda( \Bal_\lambda(t\delta_w,1)+\chi_{D(0)},1) \\
    &= \Bal_\lambda(\chi_{D'(t)}+\chi_{D(0)},1)\\
    &= \chi_{D(t)}. \qedhere
  \end{align*}
\end{proof}
\noindent
Finally, we can verify that elliptic growth coincides with elliptic balayage.
\begin{theorem}
  If a family of domains $\{D(t):0\leq t<t_0\}$ satisfies \eqref{EG_definition} with $Q>0$, then
  \begin{equation}\label{weak_EG}
    \Bal_\lambda(\chi_{D(0)}+Qt\delta_w, 1) = \chi_{D(t)}.
  \end{equation}
\end{theorem}
\begin{proof}
  As in the proof of Richardson's theorem, for a smooth function $\phi$ we have
  \begin{equation*}
    \frac{d}{dt}\int_{D(t)} \phi\, dA = Q\phi(w) + \int_{D(t)} g_w L\phi\, dA.
  \end{equation*}
  Integrating over time and choosing $\phi(z) = E(z-a)$ with $a \in D(t)^c$ yields
  \begin{equation*}
    \Lambda^{D(t)} - \Lambda^{D(0)} = QtE(w-a) = \Lambda^{Q\delta_w},
  \end{equation*}
  so that outside of $D(t)$ we have $\Lambda^{D(t)} = \Lambda^{D(0)+Qt\delta_w}$.
  More generally, choosing $\phi(z) = E(z-a)$ with any pole $a$ gives a `subharmonic' function---that is, $L\phi\geq 0$. Then
  $g_wL\phi \leq 0$ and we have
  \begin{equation*}
    \Lambda^{D(t)} - \Lambda^{D(0)} \geq QtE(w-a) = \Lambda^{Q\delta_w},
  \end{equation*}
  so that $\Lambda^{D(t)} \geq \Lambda^{D(0) + Qt\delta_w}$ throughout $\C$.
  
  Since $L\Lambda^{D(t)} = \chi_{D(t)} \leq 1$, the potential $\Lambda^{D(t)}$ satisfies \eqref{elliptic_balayage_definition}
  with $\mu=\chi_{D(0)} + t\delta_w$.
  It remains to show that $\Lambda^{D(t)}$ is the smallest of all functions satisfying \eqref{elliptic_balayage_definition};
  let $V$ be one such function. Note that $\Lambda^{D(t)} - V\leq \Lambda^{D(t)} - \Lambda^{D(0)+t\delta_w} = 0$ on $D(t)^c$
  and $L(\Lambda^{D(t)}-V) = 1-LV \geq 0$ inside $D(t)$. The maximum principle for subharmonic functions implies that $\Lambda^{D(t)}\leq V$
  inside $D(t)$ as well. Thus $\Lambda^{D(t)} \leq V$ everywhere, as desired.
\end{proof}
With these previous two theorems in hand, we can extend the definition of elliptic growth to a weak formulation;
we simply say that $\{D(t):0\leq t<t_0\}$ describes a weak elliptic growth process if equation \eqref{weak_EG} is satisfied.
Note that equation \eqref{weak_EG} uniquely identifies $D(t)$ up to null sets.
\begin{corollary}
  Weak elliptic growth of a domain exists and is unique up to null sets for all times $t>0$. If a strong solution exists,
  it is also unique up to null sets.
\end{corollary}

\section{An Alternative to Elliptic Growth}
A problem with the reverse--time Hele--Shaw flow---that is, suction---is the formation of cusps in finite time;
in the neighborhood of a cusp our simplifying assumptions of neglecting dependence on some spatial variables are invalid.
Thus we return to the equation $\mu\Delta v=\nabla p$, inserting our regularized $z$ dependence:
\begin{equation*}
  \mu\left(\partial_x^2 + \partial_y^2 - \frac{12}{h^2}\right)v = \nabla p.
\end{equation*}
We rewrite this as
\begin{equation} \label{Stokes_flow}
  \left(\frac{h^2}{12}\Delta - 1\right)v = \frac{h^2}{12\mu}\nabla p,
\end{equation}
where $\Delta$ now represents the two--dimensional Laplacian. Once again we view this equation in $\C$,
having removed the $z$ dependence. Nonetheless, the use of the Helmholtz operator rather than just the Laplacian
in equation \eqref{Stokes_flow} is meant to retain some of the character of three--dimensional flow. For this reason, we
call a flow governed by equation \eqref{Stokes_flow} Quasi--2D Stokes Flow (QSF, for short).

If we make the approximation $h^2/12 \ll 1$, we obtain the equation governing Hele--Shaw flow; in this way we see that
QSF is a singular perturbation of Laplacian growth. On the other hand, approximating $h^2/12 \gg 1$ we obtain traditional
Stokes flow. In this sense QSF forms an intermediate type of flow which could be used explain singularities occurring in
Laplacian growth.

We can reformulate the dynamics purely in terms of $p$, much as we have done with Laplacian and elliptic growth.
Let $a=12/h^2$, assume that $\mu=1$, and recall that the Laplacian of a vector field is defined to be
\begin{equation*}
  \Delta \mathbf{v} = \nabla(\nabla\cdot\mathbf{v}) - \nabla\times(\nabla\times\mathbf{v}).
\end{equation*}
Thus equation \eqref{Stokes_flow} becomes
\begin{equation*}
  \nabla(\nabla\cdot\mathbf{v}) - \nabla\times(\nabla\times\mathbf{v}) - a\mathbf{v} = \nabla p.
\end{equation*}
Taking the divergence, we obtain
\begin{equation*}
  (\Delta - a)(\nabla\cdot \mathbf{v}) = \Delta p.
\end{equation*}
If we assume the fluid is incompressible with a point source at 0, we have $\nabla\cdot\mathbf{v} = \delta_0$.
The motion of the resulting fluid domain $D$ is given by
\begin{equation} \label{QSF_growth}
  \begin{cases}
  \Delta p = (\Delta - a)\delta_0 & \textrm{ in } D\\
  p = 0 & \textrm{ on } \partial D, \\
  \end{cases}
\end{equation}
with velocity field given by $(\Delta - a)\mathbf{v} = \nabla p$. This model opens a new avenue for the study of pressure--driven
quasi--static fluid flow in a narrow channel, providing a direction for future research. Difficulties abound, however;
QSF is mathematically delicate, involving more derivatives than elliptic growth. Furthermore the model
resists the more physical reasoning we used for elliptic growth---it is difficult to compute even simple examples and
no analogue of the Richardson theorem is apparent.

\section*{Acknowledgements}
  The author wishes to thank both Mihai Putinar and Erik Lundberg for many helpful discussions on the
  mathematics of elliptic growth.
  The author also wishes to thank Mark Mineev--Weinstein for fruitful discussions of
  Hele--Shaw flow and reregularizations, from which QSF emerged; we present it here to ensure it finds a place in the literature.

\bibliographystyle{amsplain}

\vspace{0.2in}
\textsc{Department of Mathematics, Vanderbilt University, Nashville, TN, 37240} \\
\url{charles.z.martin@vanderbilt.edu}

\end{document}